\documentclass[11pt]{amsart}
\usepackage{geometry}                
\usepackage{graphicx}
\usepackage{amssymb}
\usepackage{epstopdf}
\usepackage{tikz}
\usepackage{lscape}
\usepackage{color}
\usepackage{todonotes}
\usepackage{mathrsfs}
\usepackage{amsmath}
\usepackage[utf8]{inputenc}
\usepackage{float}
\usepackage{wrapfig}
\restylefloat{figure}
\usepackage{amsthm}
\usepackage{mathdots}
\usepackage{listings}
\usetikzlibrary{arrows.meta}
\usetikzlibrary{positioning}
\usepackage{times}
\usepackage{diagbox}
\usepackage{url}
\usepackage{tikz-3dplot}
\usetikzlibrary{arrows.meta}
\usepackage{amsrefs}
\usepackage{enumitem}
\usepackage{longtable}

\theoremstyle{plain}
\newtheorem{theorem}{Theorem}[section]

\newtheorem{corollary}[theorem]{Corollary}

\newtheorem{lemma}[theorem]{Lemma}

\theoremstyle{definition}
\newtheorem{definition}[theorem]{Definition}
\newtheorem{example}[theorem]{Example}
\newtheorem{remark}[theorem]{Remark}

\newcommand{\I}{\mathcal{I}}

\newcommand{\N}{N}

\newcommand{\Z}{\mathbb{Z}}

\newcommand{\R}{\ensuremath \mathbb{R}}
\renewcommand{\P}{\ensuremath \mathcal{P}}

\newcommand{\ra}{\rightarrow}

\renewcommand{\min}{\mathrm{min}}

\newcommand{\rk}{\mathrm{rk}}

\title[Stirling numbers and Kazhdan-Lusztig Polynomials]{Stirling numbers in Braid matroid Kazhdan-Lusztig Polynomials}
\author{Trevor K. Karn}
\author{Max D. Wakefield}
\thanks{}
\address{}
\email[]{}

\begin{document}
\maketitle

\begin{abstract}

Restricted Whitney numbers of the first kind appear in the combinatorial recursion for the matroid Kazhdan-Lusztig polynomials. In the special case of braid matroids (the matroid associated to the partition lattice, the complete graph, the type A Coxeter arrangement and the symmetric group) these restricted Whitney numbers are Stirling numbers of the first kind. We use this observation to obtain a formula for the coefficients of the Kazhdan-Lusztig polynomials for braid matroids in terms of sums of products of Stirling numbers of the first kind. This results in new identities between Stirling numbers of the first kind and Stirling numbers of the second kind, as well as a non-recursive formula for the braid matroid Kazhdan-Lusztig polynomials.

\end{abstract}

\section{Introduction}

The Kazhdan-Lusztig polynomial of a matroid $M$ is recursively defined in \cite{klpoly_epw} using the lattice of flats of $M$. These polynomials mimic the classic Kazhdan-Lusztig polynomials for Coxeter groups, discovered and studied by Kazhdan and Lusztig in \cite{kl79}. The combinatorial formulation of both the classic Kazhdan-Lusztig and the matroid Kazhdan-Lusztig polynomials are special instances of a wider framework built in \cite{St-92} and \cite{Br-03} ($R$ and $\chi$ are the $P$-kernels of the Kazhdan-Lustig-Stanley polynomials respectively). In the cases of realizable matroids and Weyl groups Kazhdan-Lusztig polynomials have significant geometric interpretations; both are polynomials the Poincar\'e polynomials of the intersection cohomology of the reciprocal plane and Schubert varieties, respectively (see \cite{klpoly_epw} and \cite{KL-80}).

Arguably, the most important family of matroids is the braid matroids $\{B_n\}_{n\geq 2}$. They are graphic matroids corresponding to the complete graphs. They are also the matroids associated to the Type A Coxeter groups (a.k.a. the symmetric groups). We call them braid matroids because the complex complement of the Type A Coxeter arrangement is exactly the configuration space of $n$ points in $\R^2$ whose fundamental group is the pure braid group. Furthermore, the classical Kazhdan-Lusztig polynomials are most well known for the Type A Coxeter groups (see \cite{BB-05} and \cite{Warr-11}).

At present, there is no known relationship between the classic Kazhdan-Lusztig polynomials on the Type A Coxeter groups and the matroid Kazhdan-Lusztig polynomials for braid groups other than that they come from $P$-kernels and are instances of Kazhdan-Lusztig-Stanley polynomials \cite{Br-03}. On the other hand, in \cite{proudfoot_preprint_17} Proudfoot works towards unifying the results from \cite{Br-03} and \cite{klpoly_epw}. In particular he shows that if a $P$-kernel satisfies a specific point-counting formula, then the associated Kazhdan-Lusztig-Stanley polynomial yields the intersection cohomology.

In \cite{klpoly_epw}, Elias, Proudfoot, the second author, and Young compute many coefficients of the polynomial for braid matroids, present a generating function identity, and give a formula for the first three degree coefficients in terms of Stirling numbers of the first and second kinds. Additional formulas for arbitrary matroids in terms of ``flag'' Whitney numbers of the second kind are presented in \cite{whitney} and \cite{z-polynomial_proudfootxuyoung}. In the case of braid matroids these formulas can be presented in terms of products of Stirling numbers of the second kind. 

In \cite{GPY-16} Gedeon, Proudfoot, and Young define the matroid equivariant Kazhdan-Lusztig polynomial (which is actually a virtual representation of the group acting on the matroid) and compute many coefficients for braid matroids. In \cite{GPY-17} the same authors revisit these examples and give a conjecture for the leading coefficient when matroid rank is even. Then in \cite{PY-17} Proudfoot and Young study the equivariant Kazhdan-Lusztig polynomials of braid matroids exclusively. There they interpret these virtual representations as $FS^{op}$ modules and present an asymptotic formula for their dimensions.

In this paper we present a combinatorial formula, Theorem \ref{main_theorem}, in terms of Stirling numbers of the first kind (rather than the second kind), for the (non-equivariant) Kazhdan-Lusztig polynomial coefficients for braid matroids. In order to present the formula we need an index set which contains special sequences of number partitions. The structure of this index set reveals the recursive definition in plain sight when compared to the formulas in \cite{whitney} and \cite{z-polynomial_proudfootxuyoung}. On the other hand the number of terms in Theorem \ref{main_theorem} is larger than those in \cite{whitney} and \cite{z-polynomial_proudfootxuyoung}.

An immediate corollary from Theorem \ref{main_theorem}, using \cite{whitney} and \cite{z-polynomial_proudfootxuyoung}, is an identity between Stirling numbers of the first and second kind. Using a result in \cite{adamchik_stirling}, Theorem \ref{main_theorem} gives a non-recursive formula for the coefficients of the Kazhdan-Lusztig polynomials for braid matroids. 

In Section \ref{matroids_more} we review some basic matroid theory, matroid Kazhdan-Lusztig polynomials, and results specific to braid matroids. The main theorem is then proved in Section \ref{mainthm}. In Section \ref{sec_corollaries} we list a few corollaries to Theorem \ref{main_theorem}. Finally we conclude with an appendix to list the main formula for certain small dimensions.

\section{Matroids and more} \label{matroids_more}

A matroid, $M$, is a collection of subsets $\I$ from a ground set $E$ such that  \begin{enumerate}
\item[(M1)] \label{matroidaxiomone} For any set $A \in \I$, if $B \subseteq A$, then $B \in \I$.

\item[(M2)] \label{matroidaxiomtwo} For $A,B \in \I$, with $|A|>|B|$, then there is an element $a \in  A \backslash B$ such that $B \cup \{a\} \in \I$.
\end{enumerate} For a general reference on matroids, see \cite{oxley, welsh}. There are many different ways to represent matroids, but in this paper we view a matroid as its lattice of flats, $(L(M), \leq)$, where
$$L(M):= \{F \subseteq E : \forall x \notin F, \rk(\{F \cup \{x\}) > \rk F \}$$ and $\rk F=\mathrm{max}\{|A|: A\subseteq F \text{ and } A\in \I\}$ with order, $\leq$, given by set inclusion. The classical \text{ M\"obius function} $\mu : L(M)^2 \ra \Z $ is defined recursively for all $x_1 \leq x_2 \in L(M)$ by $\mu(x_1,x_2)=1$ if $x_1 = x_2$ and for $x_1 < x_2$ by \[ \sum \limits_{x_1 \leq a\leq x_2} \mu ( x_1, a) = 0\] and $\mu(x,y)=0$ if $x\not\leq y$. The \text{characteristic polynomial} $\chi (L,t)$ of a finite, rank $\ell$, lattice $L$ is \[ \chi(L,t) := \sum \limits_{x \in L} \mu (\hat{0}, x) t^{\ell -\rk(x)}\] (see \cite{stan_ecomb}). When $L$ is clear from context, we will write $\chi (t)$.

For $F \in L(M)$, the restriction at $F$ is the lattice $(M^F, \leq)$ where $M^F := \{x \in L(M) : x \geq F \}$. The localization at $F$ is $(M_F, \leq)$ with $M_F := \{x \in L(M) : x \leq F\}$. 

\begin{definition}[Theorem 2.2, \cite{klpoly_epw}]\label{klpoly_defn} Let $M$ be a matroid. Let $L(M)$ be the lattice of flats of $M$. The \text{Kazhdan-Lusztig matroid polynomial}, $P_M(t)\in \Z[t]$, is the unique polynomial that satisfies the following conditions:
\begin{enumerate}
\item If $\rk(M) = 0$, then $P_M(t)=1$.
\item If $\rk(M) > 0$, then $\deg(P_M(t))< \frac{\rk(M)}{2}$.
\item \label{ax_tre}  $t^{\rk(M)}P_M(t^{-1}) = \sum \limits_{F\in L(M)} \chi(M_F, t)P_{M^F}(t).$
\end{enumerate}
\end{definition}

The $n$-dimensional braid arrangement $\mathcal{B}_n$ is the collection of all hyperplanes $H_{i,j}$, where $H_{i,j}$ is the set of all points $(x_1, x_2, ..., x_n) \in \R^n$ such that $x_i = x_j$ for $1 \leq i < j \leq n$. We call the matroid associated to the braid arrangement the braid matroid, which we write as $B_n$. It is well known that the intersection lattice of $\mathcal{B}_n$ is isomorphic to the partition lattice $\mathcal{P}(n)$ of $[n]:=\{1,2,...,n\}$ (see \cite{OT_hyperplanes}). For a set partition $F  =\{p_1, p_2, ..., p_k\} $ of $[n]$, we denote the number of blocks as $\ell(F) = k$, and define $b_i := |p_i|$. The localization of the partition lattice at $F$ has the following useful property
\begin{equation}\label{localization_property} \mathcal{P}(n)_F \cong \mathcal{P}(b_1) \times \mathcal{P}(b_2) \times \cdots \times \mathcal{P}(b_k). \end{equation} Also, the restriction of the partition lattice at $F$ has the useful property that 
\begin{equation} \label{restriction_property} \mathcal{P}(n)^F \cong \mathcal{P}(\ell(F)).\end{equation}Since $\chi$ factors over products ($\chi(X \times Y, t) = \chi(X, t) \chi(Y,t)$), by equations (\ref{localization_property}) and (\ref{restriction_property}) the following are true

\begin{enumerate}
\item[(CP1)] $\chi (\mathcal{P}(n)_F, t) \cong \prod \limits_{i=1}^k \chi(\mathcal{P}(b_i), t) $,  \label{ceepee_one}

\item[(CP2)] $\chi (\mathcal{P}(n)^F, t) \cong \chi(\mathcal{P}(k), t) $. \label{ceepee_two}
\end{enumerate}


For a number partition $\lambda \vdash n$, we will denote its \emph{blocks} by $b_i$ and the number of blocks by $\ell (\lambda)$ so that $\sum_{i=1}^{\ell(\lambda)} b_i = n$. We denote the number of set partitions whose blocks have size given by $\lambda$, as $m(\lambda)$. The explicit formula is known to be
\[ m( \lambda )= \frac{n!}{ \prod^{\ell ( \lambda)}_{i=1} b_i ! \cdot \prod^{b_1}_{j=1} ( b^t_j - b^t_{j+1} )!}, \] where $b^t_{\ell(\lambda)+1} :=0$.

Let $C_{n,i}$ be the Kazhdan-Lusztig polynomial coefficients of degree $i$ for $B_n$, that is \begin{equation*} P_n(t):=  P(B_n, t) = \sum \limits_{i \geq 0} C_{n,i}t^i.\end{equation*}

In \cite{klpoly_epw}, $C_{n,i}$ was computed for $1 \leq i \leq3$ in terms of Stirling numbers of the second kind, which are coefficients, $S(n,k)$, of the summation \[t^n = \sum \limits_{k=0}^{n} S(n,k) (t)_k . \] The Stirling numbers of the first kind $s(n,k)$ are the degree $k$ coefficients in \[ (t)_n = \frac{t!}{(t-n)!} = \sum_{k=0}^{n} s(n,k) t^k .\] 
Then the characteristic polynomial for $\mathcal{P}(n)$ is
\begin{equation}\label{chipn} \chi_n(t) := \chi( \mathcal{P}(n), t) =  \frac{(t)_n}{t} = \sum \limits_{k=1}^n s(n,k) t^{k-1}. \end{equation}

\section{The main theorem} \label{mainthm}

In order to set up the main theorem, we first define an index set, over which we later sum. 
\begin{definition} \label{index_set_k_defn}
Let $n \geq 2$ and $i < \frac{n-1}{2}$. Define $\mathcal{K}_{n,i}$ to be the set of all triples $(\Lambda, A,\Xi)$ where $\Lambda = [\lambda_1, ..., \lambda_q]$ is a sequence of number partitions, and $A = [\alpha_1, ..., \alpha_q]$ and $\Xi = [\xi_1, ..., \xi_q]$ are sequences of integers which satisfy:
\begin{enumerate}[label = (\roman*)]
\item $\lambda_1 \vdash n$  \label{firstlbda}
\item $\lambda_j \vdash \ell(\lambda_{j-1})$ for all $1 < j \leq q$ \label{second_lbda}
\item \label{five_lbda} $\alpha_1 + \xi_1 = n - 1 - i$
\item $\alpha_j + \xi_j = \ell( \lambda_{j-1}) - 1 - \xi_{j-1}$ for $j>1$ \label{six_lbda}
\item $0 \leq \alpha_j \leq |\lambda_j| - \ell ( \lambda_j)$ for all $j$ \label{seven_lbda}
\item \label{eight_lbda} $\xi_j = 0$ when $\ell(\lambda_j) = 1 $
\item $0 \leq \xi_j < \frac{\ell(\lambda_j)-1}{2}$ when $\ell(\lambda_j) \geq 2$ \label{last_lbda}
\item \label{v} $\xi_j =0$ if and only if  $q = j$.
\end{enumerate}
\end{definition}
Note that this definition also stipulates that $\lambda_j \neq 1$ $\forall j$, because of \ref{second_lbda}, \ref{eight_lbda}, and \ref{v}.

\begin{example}\label{base_case}
For $n =2$ we compute $\mathcal{K}_{n,i}$. The only possible $i$ is zero. By axioms \ref{firstlbda}-\ref{v}, the possible $\Lambda$'s are $[2]$ and $[1+1,2]$. For $\Lambda = [2]$, by \ref{five_lbda}, $\alpha_1 + \xi_1 = 1$,  by \ref{eight_lbda}, $\xi_1 = 0$ so $\alpha_1 = 1$, and all other axioms are satisfied. For $\Lambda = [1+1,2]$, by \ref{five_lbda} we have $\alpha_1 + \xi_1 = 2-1-0 = 1$. By \ref{last_lbda}, $\xi_1 =0$ since  $\ell (1+1) = 2$, which implies that $\alpha_1 = 1$. However, by \ref{seven_lbda}, we have $\alpha_1 \leq |\lambda_1| - \ell(\lambda_1) = 0$, a contradiction, which means there will be no triple with $\Lambda = [1+1,2]$. Hence, $ \mathcal{K}_{2,0}= \{([2],[1],[0])\}$.

\end{example}

While $\mathcal{K}_{n,i}$ is a somewhat convoluted index set, it allows us to write the following theorem quite succinctly.

\begin{theorem}\label{main_theorem}
For $n \geq 2$ and $i < \frac{n-1}{2}$, 
\begin{equation*} \label{ans1}
C_{n,i} = \sum \limits_{(\Lambda, A, \Xi)} \left [  \prod \limits_{j= 1}^{q} m(\lambda_j) \sum \limits_{(d^j_k) } \prod \limits_{k=1}^{\ell(\lambda_j)} s(b_k^j, d^j _k ) \right ]
\end{equation*}
where $(\Lambda,A, \Xi) = ([\lambda_1, ..., \lambda_q], [\alpha_1, ... \alpha_q], [\xi_1,...\xi_q]) \in \mathcal{K}_{n,i}$,  $b_k^j$ is the $k^{\text{th}}$ block of $\lambda_j$, and the last sum is over all sequences $(d^j_k)= (d_1^j, ..., d_{\ell(\lambda_j)}^j)$ satisfying $\sum_{k=1}^{\ell(\lambda_j)} d_k^j = \alpha_j + \ell (\lambda_j)$ and $1 \leq d_k^j \leq b_k^j$.
\end{theorem}

\begin{proof} The essence of the proof is demonstrating a bijection from the index set $\mathcal{K}_{n,i}$ of Definition \ref{index_set_k_defn} to the decomposition (in terms of Stirling numbers of the first kind) of the recursive Definition \ref{klpoly_defn}. To present this bijection, we decompose the recursion step by step through induction on $n$, while simultaneously recording each axiom of Definition \ref{index_set_k_defn} as the axiom is used. 

First, note that $\rk(B_1) = 0$, hence by Definition \ref{klpoly_defn}, $P_1(t) = 1$. For $n = 2$, from Example \ref{base_case}, we know, $\mathcal{K}_{2,0}=\{([2],[1],[0])\}$, hence our formula gives 
\[C_{2,0}= m(2) s(2,2)=1.\]
This is confirmed by \cite{klpoly_epw}, as it is known that $C_{n,0} = 1$ for all $n$. 

Now for $n > 2$,  we will look at each term of the sum in the defining recursion, 
\begin{equation} \label{first_proof_step} t^{n-1}P_n(t^{-1}) = \sum \limits_{F\in L(\mathcal{B}_n)}   \chi((\mathcal{B}_n)_F, t) P((\mathcal{B}_n)^F,t). \end{equation}
By Equation (\ref{restriction_property}) and (CP1), the right hand side of (\ref{first_proof_step}) becomes
\begin{equation}\label{second_proof_step} \sum \limits_{F \in L(\mathcal{B}_n)} \left [  \prod \limits_{k=1}^{\ell(F)} \chi_{b_k}(t) \right ] P_{\ell(F)}(t) \end{equation} where $b_k$ are the size of the blocks of $F$, and $\ell(F)$ is the number of blocks.

Notice that the terms of (\ref{second_proof_step}) depend only on the size and number of blocks of $F$. This allows us to turn (\ref{second_proof_step}) into a sum over number partitions, provided we are careful to count all of the set partitions which correspond to each number partition
\begin{equation}\label{third_proof_step} \sum \limits_{\lambda_1 \vdash n} m( \lambda_1) \left [ \prod \limits_{k=1}^{\ell(\lambda_1)}  \chi_{b_k}(t) \right ] P_{\ell(\lambda_1)}(t). \end{equation} This is the reason for Definition \ref{index_set_k_defn} \ref{firstlbda}. 

To compute the coefficients $C_{n,i}$, we compute the degree $n-1-i$ coefficient from (\ref{third_proof_step}). 
For a polynomial $f = f(t) $, we write $f\{ \alpha_k\}$ to mean the coefficient of the $t^{\alpha_k}$ term. Also, for any integer $N$, and $\lambda_j \vdash N$, we write the partition and its blocks as follows,
\[\lambda_j = b_{1}^j + b_2^j + \cdots + b_{\ell(\lambda_j)}^j,\]
with $b_1^j \geq b_2^j \geq \cdots \geq b_{\ell(\lambda_j)}^j$.

The $n-1-i$ coefficient of (\ref{third_proof_step}) is
\begin{equation} \label{fourth_proof_step}
\sum \limits_{\substack{(\lambda_1 \vdash n, \alpha_1, \xi_1)\\ \alpha_1 + \xi_1 = n-1 -i}} \hspace{-.15 in}  m(\lambda_1) \left ( \left (\prod \limits_{k=1}^{\ell(\lambda_1)} \chi_{b_k^1}  \right )  \{ \alpha_1 \} \right ) P_{\ell(\lambda_1)}  \{\xi_1 \}.
\end{equation}

This gives us $\alpha_1$ and $\xi_1$ of Definition \ref{index_set_k_defn} \ref{five_lbda}. We also get the reason for  Definition \ref{index_set_k_defn} \ref{seven_lbda}, \ref{eight_lbda}, and \ref{last_lbda} with $j=1$ because of the possible degrees of the two polynomials in (\ref{fourth_proof_step}).  Then $q=1$, if and only if we take the constant coefficient of $P_{\ell(\lambda_1)}$, so $\xi_1 = 0$ (otherwise we would need to further decompose $P_{\ell(\lambda_1)}$). This is the reason for Definition \ref{klpoly_defn} \ref{v} with $j=1$.

Observe that $P_{\ell(\lambda_1)} \{\xi_1 \}$ is exactly the Kazhdan-Lusztig coefficient $C_{\ell(\lambda_1), \xi_1}$.
We also know from equation (\ref{chipn}) that the terms coming from  
$\left (\prod \limits_{k}^{\ell(\lambda_1)} \chi_{b_i^1}\right )  \{ \alpha_1 \} $
will be a sum of products of Stirling numbers of the first kind
. This allows us to write (\ref{fourth_proof_step}) as 
\begin{equation} \label{fifth_proof_step}
\sum \limits_{\substack{(\lambda_1 \vdash n, \alpha_1, \xi_1)\\ \alpha_1 + \xi_1 = n-1 -i}} \hspace{-.15 in} m(\lambda_1) \left ( \sum \limits_{\substack{ (d_k^1)_{k=1} ^{\ell (\lambda_1)}  \\   }} \prod \limits_{k=1}^{\ell(\lambda_1)} s( b^1_k, d^1_k) \right ) C_{\ell (\lambda_1), \xi_1}
\end{equation}
where $(d^1_k)= (d_1^1, ..., d_{\ell(\lambda_1)}^1)$ satisfies $\sum_{k=1}^{\ell(\lambda_1)} d_k^1= \alpha_1 + \ell (\lambda_1)$ and $1\leq d_k^1 \leq b_k^1$.

Proceeding with the induction, we write
\begin{equation}\label{insidecoeff}
C_{\ell(\lambda_1), \xi_1} = \sum \limits_{(\Lambda^{\prime}, A^{\prime}, \Xi^{\prime})} \left ( \prod_{j=2}^q m(\lambda_j) \sum \limits_{(d^j_k)_{k=1}^{\ell(\lambda_j)}} \prod \limits_{k=1}^{\ell(\lambda_j)} s(b_k^j , d_k^j) \right ).
\end{equation} 
where $(\Lambda^{\prime}, A^{\prime}, \Xi^{\prime}) \in \mathcal{K}_{\ell(\lambda_1), \xi_1}$, and $\Lambda^{\prime} = [\lambda_2, ..., \lambda_q]$, $A^{\prime} = [\alpha_2 , ..., \alpha_q]$ and $\Xi^{\prime} = [\xi_2, ..., \xi_q ]$. The induction hypothesis is why the remaining axioms in Definition \ref{index_set_k_defn} (\ref{second_lbda}, \ref{six_lbda}, and \ref{seven_lbda}-\ref{v} for $j \geq 2$) are satisfied.

Substituting (\ref{insidecoeff}) into (\ref{fifth_proof_step}) gives us that 
\begin{equation}\label{sixth_proof_step}
C_{n,i} = \sum \limits_{\substack{ \lambda_1 \vdash n \\ \alpha_1, \xi_1}} m(\lambda_1) \left ( \sum \limits_{(d_k^1)} \prod \limits_{k=1}^{\ell(\lambda_1)} s( b^1_k, d^1_k) \right ) \left ( \sum \limits_{(\Lambda^{\prime}, A^{\prime}, \Xi^{\prime})} \left ( \prod_{j=2}^q m(\lambda_j) \sum \limits_{(d^j_k)}\prod \limits_{k=1}^{\ell(\lambda_j)} s(b_k^j , d_k^j) \right ) \right ).
\end{equation}

Since $m(\lambda_1)$ and $(d_k^1)$ are both fixed with respect to $\Lambda^{\prime}$, we can bring the term 
\[m(\lambda_1) \sum \limits_{(d_k^1)_{k=1}^{\ell(\lambda_1)}} \prod \limits_{k=1}^{\ell(\lambda_1)} s(b_k^1, d_k^1)\]
inside the summation over $(\Lambda^{\prime}, A^{\prime}, \Xi^{\prime})$. 
Observe that a change of indices simplifies (\ref{sixth_proof_step}) into
\begin{equation}\label{eleventh_proof_step}
C_{n,i}= \sum \limits_{\substack{ \lambda_1 \vdash n \\ \alpha_1, \xi_1}}\sum\limits_{(\Lambda^{\prime}, A^{\prime}, \Xi^{\prime})} \left (\prod_{j=1}^q m(\lambda_j) \sum \limits_{(d^j_k)} \prod \limits_{k=1}^{\ell(\lambda_j)} s(b_k^j , d_k^j) \right ). 
\end{equation}
Now let $\Lambda = [\lambda_1,\lambda_2, ..., \lambda_q]$, $A  = [\alpha_1 , \alpha_2,  ..., \alpha_q]$ and $\Xi  = [\xi_1, \xi_2, ..., \xi_q ]$. \end{proof}

\begin{remark}
Observe that our $\xi$'s do not appear anywhere outside of $\mathcal{K}_{n,i}$ in Theorem \ref{main_theorem}. They act as placeholders to help compute the $\alpha$'s but could be eliminated in the formulation. In fact, it can be shown that \[ \xi_j = \sum \limits_{i=0}^{q-j-1} (-1)^i \ell( \lambda_{j+1}) + \sum \limits_{i=1}^{q-j} (-1)^j \alpha_{j+1} - \frac{1+(-1)^{q-j}}{2} .\]
\end{remark}






\section{Stirling numbers of the second kind and corollaries}\label{sec_corollaries}
In this section, we look at a few applications of Theorem \ref{main_theorem}. In \cite{z-polynomial_proudfootxuyoung} and \cite{whitney}, there are formulas given for arbitrary matroids, in terms of Whitney numbers of the second kind. Theorem \ref{main_theorem} can be viewed as a Whitney number of the first kind equivalent to these known formulas for the specific case of partition lattices and braid matroids. To make these connections clear, we recall the definition of Whitney numbers of the second kind.

\begin{definition}\label{flagdefn}
For a finite ranked poset $\P$ such that $rk(\P) = N$, and an ordered index set $I = (i_1, i_2, ..., i_r)$ with $i_j \in [\N]$ such that $i_1 \leq i_2 \leq ... \leq i_r$, the set of \text{partial flags} of $\P$ associated to $I$ is \[\P_I = \{ (X_1, X_2, ..., X_r) \in \P^r | \forall 1 \leq j \leq r, \rk(X_j) = i_j, \text{ and } X_1 \leq X_2 \leq ... \leq X_r\}. \]
The \text{ partial flag (also called multi-indexed) Whitney numbers of the second kind} are $W_I = |\P_I|.$
\end{definition}

In order to state the  next theorem we use an index set $S_i$ and two functions $s_i: S_i \rightarrow \Z$ and $t: S_i \rightarrow 2^{\Z [N]}$  constructed in \cite{whitney}.

\begin{theorem}[\cite{whitney}] \label{whitney_thm_wake}
For any finite, ranked lattice $\P$ such that $\rk(\P) = N$, the degree $i$ coefficient of the Kazhdan-Lusztig polynomial of $\P$ with $1 \leq i < N/2$ is 
\begin{equation*}
\sum \limits_{I \in S_i} (-1)^{s_i(I)} ( W_{t(I)}(\P) - W_I (\P)).
\end{equation*}
\end{theorem}

The following elementary lemma, mentioned in \cite{z-polynomial_proudfootxuyoung}, states that flag Whitney numbers of the second kind can be written as products of Stirling numbers of the second kind in the case of partition lattices. However we include it here, as well as a proof, in order to state the main corollary of this section.

\begin{lemma}\label{product_whitneystirling_lemma}
For $I = (i_1, i_2, ..., i_k)$, and setting $i_0 = 0$, we have that for the partition lattice of $[n]$ \[W_I = \prod \limits_{j = 0}^{k-1}S(n-i_j, n-i_{j+1}). \]
\end{lemma}
\begin{proof}
Induct on $k$. For $k = 1$, $W_{\{i\}} = s(n,n-i)$. For $k > 1$, we know $P(n)^F \cong P(n-rk(F))$. Hence, 
\begin{align*}
W_I &= \sum \limits_{F \in P(n)_{i_1}} W_{I \backslash \{i_1\}} P(n-i_1) \\
&= \sum \limits_{F \in P(n)_{i_1}} \prod \limits_{j=1}^{k-1} S(n-i_j, n-i_{j+1}) \\
&= \prod \limits_{j=0}^{k-1} S(n-i_j, n-i_{j+1}) \qedhere \end{align*} \end{proof}

As noted by Proudfoot, Xu, and Young in Remark 3.6 of \cite{z-polynomial_proudfootxuyoung}, Theorem \ref{whitney_thm_wake} bears strong similarity to a theorem of their own. They use a notation for the multi-indexed Whitney number (which they call the $r$-Whitney number) that orders the multi-index by corank, as opposed our notation (see Definition \ref{flagdefn}) which is ordered by rank (see \cite{z-polynomial_proudfootxuyoung}, Remark 3.1). They define it as such, because it simplifies the statement of their theorem. Here we state their theorem using our notation.

\begin{theorem}[\cite{z-polynomial_proudfootxuyoung}] \label{whitney_thm_pxy}
For all $i >0$, the degree $i$ coefficient of the matroid Kazhdan-Lusztig Polynomial, $C_i$ for a matroid of rank $N$ is 
\[ \sum \limits_{r=1}^i \sum \limits_{D \subset [r]} (-1)^{|D|} \sum \limits_{(a_m)} W_{ ( N- (a_{t_r (S)} + a_{r+1}), ..., N-(a_{t_1 (S)} + a_0 ))} \] 
where $W$ is the multi-indexed Whitney number of the sequence of integers $(a_m)$ such that $a_0=0$, $a_r = i$, $a_{r+1} = \rk(M) - i$, $a_0 < a_1 < \cdots < a_r < a_{r+1}$, and \[t_j (S) = \min\{ k | k \geq j \text{ and } k \not\in S\} \in [r+1]. \]
\end{theorem}

With Theorems \ref{main_theorem}, \ref{whitney_thm_wake}, \ref{whitney_thm_pxy}, and Lemma \ref{product_whitneystirling_lemma} we can write the following corollary.

\begin{corollary}
For $n \geq 2$ and $i < \frac{n-1}{2}$
\begin{align*}
C_{n,i} =&\sum \limits_{I \in S_i} (-1)^{s_i(I)} \left ( \bigg (\prod \limits_{\substack{i_j \in t(I) \\ i_0=0}} S(n-i_j, n-i_{j+1})\bigg)- \bigg (\prod \limits_{\substack{i_j \in I \\ i_0=0}} S(n-i_j, n-i_{j+1})\bigg ) \right) \\
=& \sum \limits_{r=1}^i \sum \limits_{D \subset [r]} (-1)^{|D|} \sum \limits_{(a_m)} \prod \limits_{j = 0}^{r-1} S(a_{t_j(D)} + a_{j+1}, a_{t_{j+1}(D)} + a_{j+1} ) \\
=&  \sum \limits_{(\Lambda, A, \Xi)} \left [  \prod \limits_{j= 1}^{q} m(\lambda_j) \sum \limits_{(d^j_k) } \prod \limits_{k=1}^{\ell(\lambda_j)} s(b_k^j, d^j _k ) \right ],
\end{align*}
when all the conditions from Theorems  \ref{main_theorem}, \ref{whitney_thm_wake}, and \ref{whitney_thm_pxy} are satisfied.
\end{corollary}

We now state this for the specific case of $i=1$ for all $n$.

\begin{corollary}\label{corolla_two} For $n \geq 2$,
\[C_{n,1} = s(n,n-1)+\sum_{\substack{\lambda = \lambda_1  + \lambda_2 \\ \lambda \vdash n} } m(\lambda) = S(n,2)-S(n,n-1)\]
where the sum is over partitions of $n$ into two blocks. \end{corollary}

Another interesting fact, known from \cite{adamchik_stirling}, is that that Stirling numbers of the first kind can be written non-recursively as 
\[ s(n,m) = \frac{(n-1)!}{(m-1)!} \sum \limits_{i_1 = 1}^{n-1} \sum \limits_{i_2 = i_1 + 1}^{n-1} \cdots  \sum \limits_{i_m = i_{m-1} + 1}^{n-1}  \frac{m!}{i_1 i_2 \dots i_m}. \]

\begin{corollary} For $n \geq 2$ and $i < \frac{n-1}{2}$, the following non-recursive formula holds
\[ C_{n,i} = \sum \limits_{(\Lambda, A, \Xi)} \left [  \prod \limits_{j= 1}^{q} m(\lambda_j) \sum \limits_{(d^j_k) } \prod \limits_{k=1}^{\ell(\lambda_j)} \left ( \frac{(b_k^j-1)!}{(d^j _k-1)!} \sum \limits_{i_1 = 1}^{b_k^j-1} \sum \limits_{i_2 = i_1 + 1}^{b_k^j-1} \cdots  \sum \limits_{i_{d^j _k} = i_{d^j _k-1} + 1}^{b_k^j-1}  \frac{d^j _k!}{i_1 i_2 \dots i_{d^j _k}} \right ) \right ]. \]
\end{corollary} 
We conclude with a table listing various values of $C_{n,i}$.

	\begin{longtable}[t]{|c||l|}
	\hline
	$C_{2,0}$ & $m(2)s(2,2)$ \\
	\hline \hline
	$C_{3,0}$ & $m(3)s(3,3)$\\
	\hline \hline
	$C_{4,0}$ & $m(4)s(4,4)$\\
	\hline
	$C_{4,1}$ & $m(4)s(4,3)+m(3+1)s(3,3)s(1,1)m(2)s(2,2)+m(2+2)s(2,2)s(2,2)m(2)s(2,2)$\\
	\hline \hline
	$C_{5,0}$ & $m(5)s(5,5)$\\
	\hline
	$C_{5,1}$ & \shortstack[l]{\\ $m(5)s(5,4)+m(3+2)s(2,2)s(3,3)m(2)s(2,2) +m(4+1)s(4,4)s(1,1)m(2)s(2,2)$}    \\
	\hline \hline
	$C_{6,0}$ & $m(6)s(6,6)$\\
	\hline
	$C_{6,1}$ & \shortstack[l]{ \\ $m(6)s(6,5) + m(5+1)s(5,5)s(1,1)m(1)m(2)s(1,1)s(2,2)  $ \\ $+ m(4+2) \big (m(2)s(2,2)s(4,4)s(2,2) \big)$\\$ +m(3+3) \big ( m(2)s(2,2) s(3,3)s(3,3) \big )$}   \\
	\hline 
	$C_{6,2}$ & \shortstack[l]{$m(6)s(6,4) + m(5+1)s(5,3)m(2)s(2,2) + m(4+2) \big ( s(4,4)s(2,2)m(2)s(2,2) \big )$\\
					$ + m(3+3) \big (2 m(2) s(2,2) s(3,2) s(3,3)  \big ) $\\ $+m(4+1+1) s(1,1)s(1,1)s(4,4) m(3) s(3,3)$ \\ 
					$+m(3+2+1)s(1,1)s(2,2)s(3,3)m(3)s(3,3)$ \\ $+m(2+2+2)s(2,2)s(2,2)s(2,2) m(3)s(3,3)$ \\
					 $+m(3+1+1+1) s(1,1)^3 s(3,3) \big (  m(4)s(4,3) + m(3+1) s(3,3) s(2,2) \big )$ } \\
	\hline \hline
	$C_{7,0}$ & $m(7)s(7,7)$\\
	\hline
	$C_{7,1}$ & \shortstack[l]{$ m(7)s(7,6) + m(6+1)s(6,6)s(2,2)$\\ $+m(5+2)s(5,5)s(2,2)s(2,2) + m(4+3)s(4,4)s(3,3)s(2,2)$} \\
	\hline
	$C_{7,2}$ & \shortstack[l]{  \\
	$ m(7)s(7,5) + m (1+6) \big ( s(6,5)s(1,1)m(2)s(2,2) \big )$ \\ 
$+ m(5+2) \big ( s(5,5)m(2)s(2,2)s(2,1) + s(5,4)s(2,2)m(2)s(2,2) \big ) $\\
$+ m(4+3) \big ( s(4,4)s(3,2)m(2)s(2,2) + s(4,3)s(3,3)m(2)s(2,2)\big) $\\
$+ m(5+1+1) \big ( s(5,5)s(1,1)^2m(3) s(3,3) \big ) $\\
$+ m(4+2+1) \big ( s(4,4) s(2,2) s(1,1)m(3) s(3,3) \big )$ \\
$+ m(3+3+1) \big ( s(3,3)s(3,3)s(1,1)m(3)s(3,3) \big ) $\\
$+ m(3+2+2) \big ( s(3,3)s(2,2)s(2,2)m(3)s(3,3) \big) $\\
$+ m(4+1+1+1) \big (s(4,4)s(1,1)^3 (m(4)s(4,3) + m(3+1)s(1,1)s(3,3)m(2)s(2,2) +$ \\ 
                             $\hspace{1.35 in} m(2+2)s(2,2)^2 m(2)s(2,2) \big ) $\\
$+ m(3+2+1+1) \big ( s(3,3)s(1,1)^2s(2,2) ( m(4)s(4,3) + m(3+1)s(1,1)s(3,3)m(2)s(2,2) +$ \\ 
                             $\hspace{1.35 in} m(2+2)s(2,2)^2 m(2)s(2,2) \big ) $\\
$+ m(2+2+2+1) \big ( s(2,2)^3 s(1,1)  ( m(4)s(4,3) + m(3+1)s(1,1)s(3,3)m(2)s(2,2) + $\\ 
                             $\hspace{1.35 in} m(2+2)s(2,2)^2 m(2)s(2,2) \big )$}\\
	\hline
	\end{longtable}

\setcounter{table}{0}

\color{black}


\bibliographystyle{amsalpha}
\bibliography{kl_stirling-sources}

\end{document}